\documentclass[12pt,a4paper,leqno]{amsart}
\usepackage[utf8]{inputenc}
\usepackage[T1]{fontenc}
\usepackage[english]{babel}
\usepackage{amsmath}
\usepackage{amsfonts}
\usepackage{amssymb}
\usepackage{amsthm}
\usepackage{times, fullpage}
\usepackage{mathrsfs}
\usepackage{graphicx}
\usepackage{braket}
\usepackage{color}
\usepackage{amscd,epsfig,esint,graphics}
\usepackage{wrapfig}
\usepackage{verbatim}
\usepackage{mathtools}
\usepackage[dvipsnames]{xcolor}
\usepackage[colorlinks, allcolors=RedViolet,pdfstartview=,pdfpagemode=UseNone]{hyperref}

\theoremstyle{plain}
\newtheorem{theorem}[equation]{Theorem}
\newtheorem{lemma}[equation]{Lemma}

\theoremstyle{definition}

\theoremstyle{remark}
\newtheorem{example}[equation]{Example}
\newtheorem{remark}[equation]{Remark}

\numberwithin{equation}{section}

\newcommand{\ainc}[1]{\hyperref[ainc]{{\normalfont(aInc){\ensuremath{_{#1}}}}}}
\newcommand{\adec}[1]{\hyperref[adec]{{\normalfont(aDec){\ensuremath{_{#1}}}}}}
\newcommand{\inc}[1]{\hyperref[inc]{{\normalfont(Inc){\ensuremath{_{#1}}}}}}
\newcommand{\dec}[1]{\hyperref[dec]{{\normalfont(Dec){\ensuremath{_{#1}}}}}}

\renewcommand{\phi}{\varphi}
\renewcommand{\epsilon}{\varepsilon}
\renewcommand{\rho}{\varrho}
\def\le{\leqslant}

\def\ge{\geqslant}

\def\phi{\varphi}
\def\rho{\varrho}
\def\vartheta{\theta}

\newcommand{\N}{\mathbb{N}}
\newcommand{\R}{\mathbb{R}}
\newcommand{\Rn}{{\mathbb{R}^n}}

\DeclareMathOperator*{\argmax}{arg\,max}

\renewcommand{\div}{\operatorname{div}}

\definecolor{dg}{rgb}{0.01, 0.75, 0.24}

\title{Monotone approximation of differentiable convex functions with applications 
to general minimization problems}

\author{Petteri Harjulehto}
\address{Petteri Harjulehto,
Department of Mathematics and Statistics,
FI-00014 University of Helsinki, Finland}
\email{\texttt{petteri.harjulehto@helsinki.fi}}

\author{Peter Hästö}
\address{Peter Hästö, Department of Mathematics and Statistics,
FI-20014 University of Turku, Finland 
\& Department of Mathematics and Statistics,
FI-00014 University of Helsinki, Finland
}
\email{\texttt{peter.hasto@helsinki.fi}}

\author[A.\,Torricelli]
{Andrea Torricelli}
\address{A.\,Torricelli, Dipartimento di Scienze Fisiche, Informatiche e Matematiche, Universit\`a degli Studi di Modena e Reggio Emilia, via Campi 213/b, 41125, Modena, Italy.}
\email{andrea.torricelli@unipr.it}

\begin{document}
\begin{abstract}
We study minimizers of non-autonomous energies with minimal growth and coercivity assumptions 
on the energy. We show that the minimizer is nevertheless the solution of 
the relevant Euler--Lagrange equation or inequality. The main tool is an extension result for 
convex $C^1$-energies. 
\end{abstract}


\maketitle

\section{Introduction and statement of the main result}

We study minimization problems in a bounded domain $\Omega\subset\Rn$.
We denote the distributional gradient of a Sobolev function $v\in W^{1,1}(\Omega)$ by $\nabla v$, and the derivative of $F:\Omega\times\Rn\to\R$ with respect to the second variable by $F'$.
For boundary values $u_0\in W^{1,1}(\Omega)$ we define
\begin{align*}
W_{u_0}^{1,1}(\Omega):=\left\{v\in W^{1,1}(\Omega): v-u_0\in W_0^{1,1}(\Omega)\right\}.
\end{align*}
Let $F:\Omega\times\Rn\to [0,\infty)$ be a Carath\'eodory function and consider the minimization problem
\begin{equation}\tag{$\min F$}
\label{eq:minF}
\min_{v\in K}\int_\Omega F(x,\nabla v)\,dx, 
\end{equation}
where $K$ is a non-empty, closed and convex subset of $W^{1,1}_{u_0}(\Omega)$. 
For simplicity and with minor loss of generality, we assume that $F(x,0)=0=F'(x,0)$ for every $x\in\Omega$.

This problem has been studied recently for autonomous energies, i.e.\ when $F(x,\xi)=F(\xi)$ 
does not depend on $x$. In the unconstrained minimization case $K=W^{1,1}_{u_0}(\Omega)$, 
Carozza, Kristensen and Passarelli di Napoli studied 
the problem first when $F$ has $(p,q)$-growth \cite{CarKP14} and then for the 
weaker assumption of superlinear growth \cite{CarKP15}. 
In the case of the obstacle problem 
(i.e.\ $K$ from \eqref{eq:obstacle}), analogous results were derived 
by Eleuteri and Passarelli di Napoli \cite{EleP22} and 
Riccò and Torricelli \cite{RicT24}, respectively. 
In a different direction Koch and Kristensen \cite{KocK_pp} generalized \cite{CarKP15} to 
higher-order derivatives for a general closed and convex $K$. Our main contribution is to 
generalize \cite{CarKP15, EleP22, RicT24} to the non-autonomous case. Moreover, we have significantly 
streamlined the proof in order to avoid extraneous assumptions in our more general setting
(e.g.\ from convolution). The most substantial improvements 
are in the approximation procedure and in substituting generalized Young measures with a direct 
argument. While in \cite{CarKP14}, the authors obtain a smooth approximation using a four stage approach, we show in Theorem~\ref{thm:approximation} that 
the necessary properties are satisfied already for the approximant of their first stage.
Not only is this a simplification, but this approximation also has some good properties that 
the four-stage approximant lacks, such as preserving the value and derivative of the 
energy at the origin, while lacking others, like $C^\infty$-smoothness. 

Next we state our main result. 
Note that the continuous differentiability and superlinearity are both 
uniform with respect to the first variable. Moreover, due to our simplified proof, we 
are able to avoid continuity assumptions in the $x$-variable. 


\begin{theorem}
\label{thm:main}
Let $\Omega\subset\Rn$, $n\ge 2$, be a bounded domain, $u_0\in W^{1,1}(\Omega)$ and $K\subset W^{1,1}_{u_0}$ be closed and convex. 
Assume that $F: \Omega\times \Rn\to [0,\infty)$ is Carath\'eodory,  
strictly convex, locally bounded and superlinear at infinity. 
Further we assume for every $x\in\Omega$ that 
$F(x,0)=0=F'(x,0)$ and $F(x,\cdot)\in C^1(\Rn)$ and that 
$F'$ is also a Carath\'eodory function.
If $F(\cdot,t \nabla w_0)\in L^1(\Omega)$ for some $t>1$ and $w_0 \in K$, then there exists a unique minimizer $u \in K$ of \eqref{eq:minF}, 
and it satisfies
\[
F^*(\cdot,F'(\cdot,\nabla u)) \in L^1(\Omega)
\quad\text{and}\quad 
F'(\cdot,\nabla u)\cdot \nabla u \in L^1(\Omega).
\]
For any $\eta\in W^{1,\infty}(\Omega)$ with $\eta+K\subset K$, we have 
\[
\int_\Omega F'(x,\nabla u)\cdot \nabla \eta \, dx \ge 0.
\]
\end{theorem}

We have two particular closed and convex sets $K$ in mind corresponding 
to unconstrained minimization and the obstacle problem.
When $K=W^{1,1}_{u_0}(\Omega)$ the condition $\eta+K\subset K$ in the theorem holds for all 
$\eta\in C^\infty_0(\Omega)$ and the last inequality of the theorem means, 
by definition, that 
\[
\div F'(x,\nabla u) = 0
\]
in a distributional sense.
When 
\begin{equation}\label{eq:obstacle}
K=
\big\{ v\in W_{u_0}^{1,1}(\Omega): v\ge\psi \text{ a.e. in }\Omega
\big\},
\end{equation}
the condition $\eta+K\subset K$ holds for all non-negative $\eta\in C^\infty_0(\Omega)$ 
and the theorem says that 
\[
\div F'(x,\nabla u) \le 0
\]
in a distributional sense.
This is the obstacle problem with obstacle $\psi: \Omega\to [-\infty, \infty)$. 


Let us conclude with some motivation for the chosen non-autonomous framework. 
Variable exponent and double phase energies have been widely studied in the past 
20 years. Both are special cases of generalized Orlicz growth \cite{ChlGSW21, HarH19}. 
We give several ``extreme'' examples of energies which utilize the general framework that we 
study in this paper. Of course, ordinary cases such as the double phase 
energy $F(x,\xi):=|\xi|^p + a(x) |\xi|^q$ with $a:\overline\Omega\to [0,\infty)$ and $q>p>1$ 
\cite{ColM15} are also included. 

\begin{example}
Some admissible energies not covered by previous research are:
\begin{enumerate}
\item
Exponential growth with coefficient $F(x,\xi):=e^{\omega(x)\, |\xi|}-1$, where 
$\omega:\overline\Omega\to I\subset\subset (0,\infty)$.
\item
Perturbed variable exponent $F(x,\xi):= |\xi|^{p(x)}\log(2+|\xi|)$, where $p:\overline\Omega\to [1,\infty)$ 
 \cite{GiaP13}.
\item
Double phase energy with nearly linear growth $F(x,\xi):= |\xi|\log(2+|\xi|) + a(x)\, |\xi|^q$,
where $a:\overline\Omega\to [0,\infty)$ and $q>1$ \cite{DefM23}.
\item
Anisotropic double phase $F(x,\xi):=|\xi|^p + a(x) |\xi_1|^q$, where 
$a:\overline\Omega\to [0,\infty)$, $\xi_1$ is the first component of $\xi$, and $q>p>1$ \cite{ChlGSW21, Has23}.
\end{enumerate}
\end{example}

Throughout the paper we always assume that $\Omega$ is a bounded domain in $\Rn$, $n \ge 2$, 
and that $K\subset W^{1,1}_{u_0}$ is closed and convex. Moreover we denote 
$\eta+K:=\left\{\eta+v: v\in K \right\}$. For $r>0$, $B_r$ is the open ball with radius 
$r$ centered at the origin.

\section{Approximating convex functions}\label{sect:approximation}

We recall the most important facts about convex functions for this paper 
and refer to \cite{NicP18} for more details. 
A function $G:\Rn\to [0,\infty]$ is \emph{convex} if 
\[
G(h\xi + (1-h) \zeta)\le 
h G(\xi) + (1-h)G(\zeta) 
\]
for all $h\in (0,1)$ and $\xi,\zeta\in \Rn$. 
\emph{Strictly convex} refers to the inequality being strict whenever $\xi \ne \zeta$. 
If a convex function $G$ is differentiable, then 
\begin{equation}\label{equ:convex_estimate}
G(\xi) \ge G(\zeta) + G'(\zeta)\cdot (\xi-\zeta)
\end{equation}
and the derivative $G'$ is \emph{monotone}, i.e.  
\[
(G'(\xi) - G'(\zeta))\cdot (\xi-\zeta)\ge 0.
\] 
In particular, $G'(\xi) \cdot \xi\ge 0$ if $G'(0)=0$. 
A strictly convex function has strictly monotone derivative. 
In particular, such a derivative is injective. \cite[Theorems~3.8.1 and 3.8.3]{NicP18}

The \textit{conjugate} $G^* : \Rn \to [0,\infty]$ of $G:\Rn\to [0,\infty]$ is defined as 
\[
G^*(z):= \sup_{\xi\in\Rn}(\xi\cdot z - G(\xi)).
\]
Here extended real-valued functions as natural since if $G:=k\,|{\cdot}|$, then 
$G^*=\infty \chi_{\Rn\setminus\overline{B_k}}$. 
\emph{Fenchel's inequality} (also known as Young's inequality)
\[
\xi\cdot z \le G(z) + G^*(\xi)
\]
holds by definition of $G^*$. 
Moreover, if $G$ is convex and differentiable, then \emph{Fenchel's identity}
\begin{equation}
	\label{FenchelId}
G^*(G'(\xi)) + G(\xi) = \xi\cdot G'(\xi)
\end{equation}
holds for all $\xi\in \Rn$ \cite[Proposition~6.1.1]{NicP18}. 
If $G$ is convex and lower semi-continuous, then $G^{**}=G$ \cite[Theorem~6.1.2]{NicP18}. 
These and other results are applied point-wise to functions of type $G:\Omega\times \Rn\to [0,\infty]$, 
cf.\ \cite[Section~2.5]{HarH19}.

We say that $G:\Omega\times\Rn\to[0,\infty]$ has \emph{linear growth} if $|G(x,\xi)|\le m(1+|\xi|)$ 
for some $m>0$ and all $(x, \xi)\in\Omega\times \Rn$ 
and that $G$ is \emph{superlinear at infinity} (also known as supercoercive) if
\[
\lim_{\xi\to \infty} \inf_{x\in\Omega}\frac{ G(x,\xi)}{|\xi|} = \infty.
\] 
We say that $G$ is \emph{locally bounded} if 
\[
M_G(r) := \sup_{x\in\Omega, \xi\in B_r} G(x,\xi) < \infty
\] 
for every $r>0$. Note that these 
conditions are all uniform with respect to $x\in\Omega$.  
The following result connects these concepts and the conjugate; although the connection is 
not new, we could not find a suitable reference for our case, so we give a simple proof based directly on the definitions which shows that no additional assumptions, 
such as convexity, are required. 
\begin{lemma}
\label{lem:superlinear}
If $G : \Omega\times\Rn \to [0,\infty]$ is locally bounded or superlinear at infinity, 
then the conjugate $G^* : \Omega\times\Rn \to [0,\infty]$ is superlinear at infinity 
or locally bounded, respectively.
\end{lemma}
\begin{proof}
With the choice $\xi:=r \frac z{|z|}$ in the supremum in the definition, we see that 
$G^*(x,z) \ge r\, |z| - M_G(r)$. If $G$ is locally bounded, then  this implies that 
$\liminf_{z\to\infty} \inf_{x\in\Omega}\frac{G^*(x,z)}{|z|}\ge r$. Since this holds for any $r>0$, 
$G^*$ is superlinear at infinity.

If $G$ is superlinear at infinity, then for any $k>0$ we can choose $r>0$  
such that $\inf_{x\in\Omega}\frac{G(x,\xi)}{|\xi|}\ge k$ for every $\xi\in \Rn\setminus B_r$. 
When $z\in B_k$, we obtain that 
\begin{align*}
G^*(x,z) 
&= 
\max\bigg\{\sup_{\xi\in B_r}(\xi\cdot z - G(x,\xi)), \sup_{\xi\in\Rn\setminus B_r}(\xi\cdot z - G(x,\xi))\bigg\}\\
&\le
\max\Big\{r\,|z|, \sup_{\xi\in\Rn\setminus B_r}\big(k\,|\xi| - \inf_{x\in\Omega} G(x,\xi)\big)\Big\} \le kr.
\end{align*}
Thus $M_{G^*}(k)\le kr$ and so $G^*$ is locally bounded.
\end{proof}

Following Carozza, Kristensen and Passarelli di Napoli \cite{CarKP14}, 
we consider a variant of the second conjugate function $F^{**}$, where the set in the supremum is 
restricted to a ball. For $k\in (0,\infty)$, $x\in \Omega$ and $\xi \in \Rn$, we define
\[
F_k(x,\xi) := \sup_{z\in \overline{B_k}} (\xi\cdot z - F^*(x,z) ).
\]
The sequence $F_k$ is increasing and 
\[
F_k(x,\cdot) \nearrow F^{**}(x,\cdot) = F(x,\cdot) 
\]
for every $x\in\Omega$ if $F$ is convex and lower-semicontinuous. 
The choices $z=\xi/|\xi|$ and $z=k \xi/|\xi|$ give the lower bounds 
\[
F_k(x,\xi) \ge \max\big\{ |\xi| - M_{F^*}(1), k\, |\xi| - M_{F^*}(k)\big\}. 
\]
This is useful when $F$ has superlinear growth at infinity, since then 
Lemma~\ref{lem:superlinear} implies that $F^*$ is locally bounded, 
i.e.\ $M_{F^*} < \infty$. On the other hand, $F^*\ge 0$ gives the upper bound 
$F_k(x,\xi)\le k|\xi|$. 

Previous studies \cite{CarKP14, CarKP15, EleP22, KocK_pp, RicT24} have applied a sequence of 
additional regularizations to $F_k$
that do not preserve the normalization $F(0)=0=F'(0)$, needed in the non-autonomous case.
To avoid this work-around, we prove directly appropriate 
regularity for $F_k$, which is of independent interest. 
This allows us to simplify the proof of Theorem~\ref{thm:main} in subsequent steps. 
 
The next theorem is our second main result and it suggests that $F_k$ is the greatest convex, $k$-Lipschitz 
minorant of $F$. Finding the $C^1$-extension of convex functions in general is a non-trivial problem 
\cite{AzaM17, AzaM19}, but the specific choice of the set $(F')^{-1}(B_k)$ allows us 
to obtain the precise Lipschitz constant of this explicit extension. 
Such extension can be applied to conditions for anisotropic energies \cite{Has23}. A less elegant 
extension with higher order regularity is considered in \cite{HasO22}. 
For simplicity we write the next result without $x$-dependence; it is later applied point-wise.

The idea for the next proof was provided by an anonymous referee in connection with $\inf$-convolution (see \cite{Roc66}); 
our original proof is provided in Theorem~\ref{thm:approximation2} 
as the technique may be of independent interest.

\begin{theorem}\label{thm:approximation}
If $F\in C^1(\Rn)$ is non-negative, strictly convex and superlinear at infinity, then 
$F_k\in C^1(\Rn)$ is $k$-Lipschitz and coincides with $F$ in the set 
$(F')^{-1}(B_k)$.
\end{theorem}
\begin{proof}
From the definitions of conjugate and $F_k$, we  calculate
\[
F_k^*(z) 
=
\sup_{\xi\in\Rn}\Big(z\cdot \xi - \sup_{y\in \overline{B_k}}(y\cdot \xi - F^*(y))\Big)
=
\sup_{\xi\in\Rn}\inf_{y\in \overline{B_k}}\big((z-y)\cdot \xi + F^*(y)\big).
\]
When $z\in \overline B_k$, the choice $y=z$ in the infimum gives $F_k^*(z)\le F^*(z)$.
Since $F$ is strictly convex and  superlinear at infinity, $F^*$ is in $C^1(\Rn)$ \cite[Theorem~6.2.4]{NicP18}. By definition $F^*$ is convex. 
Thus by \eqref{equ:convex_estimate} we have $(z-y)\cdot (F^*)'(z) + F^*(y)\ge F^*(z)$.
The left-hand side corresponds by choosing $\xi = (F^*)'(z)$ in the definition of $F_k^*$, 
and hence we have 
$F_k^*(z) \ge F^*(z)$. 
When $z \not\in \overline{B_k}$, the choice $\xi = s z$ gives  
$(z-y)\cdot \xi + F^*(y) \ge s (z-y)\cdot z\ge s(|z|-k)|z| \to \infty$ as $s\to\infty$ 
so that $F_k^* (z) = \infty$.
Thus we have proved that  $F_k^* = F^* + \infty \chi_{\Rn\setminus\overline{B_k}}$.
%

Since $F\in C^1(\Rn)$ is convex and superlinear at infinity, 
$F^*$ is strictly convex \cite[Theorem~6.2.4]{NicP18}. Furthermore, 
$\infty \chi_{\Rn\setminus\overline{B_k}}$ is convex and superlinear at 
infinity. As a sum, $F_k^*$ is then strictly convex and superlinear at 
infinity. Therefore it follows from \cite[Theorem~6.2.4]{NicP18} that 
$F_k^{**}$ is $C^1(\Rn)$ and $F_k=F_k^{**}$ since $F_k$ is convex and lower semicontinuous \cite[Theorem~6.1.2]{NicP18}. 

From the definition of $F_k$ we also estimate with the choice 
$\zeta=\xi$ that 
\[
F_k(z)-F_k(y) = 
\sup_{\xi\in \overline{B_k}}\inf_{\zeta\in \overline{B_k}}
(\xi\cdot z - \zeta\cdot y -F^*(\xi)+F^*(\zeta))
\le 
\sup_{\xi\in \overline{B_k}} \xi\cdot (z - y )
= k\,|z-y|. 
\]
By symmetry this implies the $k$-Lipschitz continuity. 

The maximum of the function $z\mapsto \xi\cdot z - F^*(z)$ in the definition of $F_k$ is obtained at 
an interior point of $B_k$ where its gradient vanishes or at a boundary point where its 
gradient is perpendicular to the boundary. In the former case, $\xi - (F^*)'(z)=0\quad\text{for }z\in B_k$ and this is also a maximum of the function $z\mapsto \xi\cdot z - F^*(z)$ over $\Rn$; in this case $F_k(\xi)=F^{**}(\xi)=F(\xi)$.
Note that $(F^*)'(z) = (F')^{-1}(z)$ by \cite[p.~257]{NicP18}
and thus $F_k$ coincides with $F$ in $(F')^{-1}(B_k)$. 
\end{proof}

Let $R>0$. It follows from $F'(x,\xi)\cdot (\zeta-\xi)\le F(x,\zeta)-F(x,\xi)$ with 
the choice $\zeta:=\xi + \frac{F'(x,\xi)}{|F'(x,\xi)|}$ that 
$|F'(x,\xi)|\le M_F(|\xi|+1)$.
Since $F$ is locally bounded, there exists $k_0\in \N$ such that $F'(x,\xi)\in B_{k_0}$ for all $x\in \Omega$ and $\xi\in B_R$. 
Thus $F_k=F$ and $F_k'=F'$ in $\Omega\times B_R$ when $k\ge k_0$. 

Suppose that $U_k\to U$ in $L^1(\Omega; \Rn)$ and define 
$\sigma, \sigma_k:\Omega \to \Rn$ as 
$\sigma_k := F_k'(\cdot,U_k)$ and $\sigma:=F'(\cdot,U)$. 
Let us show that $\sigma_k \to \sigma$ in measure on $\Omega$.
We utilize the following result. 

\begin{theorem}[Scorza Dragoni Theorem, p.~235, \cite{EkeT99}]
A mapping $f:\Omega\times \Rn\to \Rn$ is Carath\'eodory if and only if 
for every $\epsilon>0$ and $K\Subset \Omega$ there exists $K_\epsilon\Subset K$ 
such that $|K\setminus K_\epsilon|<\epsilon$ and $f|_{K_\epsilon\times \Rn}$ is continuous.
\end{theorem}

Let $\epsilon>0$ and $R:=\frac1\epsilon$ and choose $k_0>0$ as before such that 
$F_k=F$ and $F_k'=F'$ in $\Omega\times B_R$ when $k\ge k_0$. 
Since $F'$ is assumed to be Carath\'eodory and $\Omega$ is bounded, we can choose by the Scorza Dragoni Theorem 
a set $K_\epsilon\Subset\Omega$ such that 
$|\Omega\setminus K_\epsilon|<\epsilon$ and $F'|_{K_\epsilon\times \Rn}$ is continuous.
By uniform continuity in the compact set, we can choose $\delta>0$ such that 
$|F'(x,\xi')-F'(x,\xi)|<\epsilon$ for all $(x, \xi')\in K_\epsilon \times\overline{B_R}$ and 
$|\xi'-\xi|<\delta$. Since $U_k\to U$ in measure, we can choose $k_1>0$ 
such that $|\{ |U_k-U|\ge\delta\}|<\epsilon$ for all $k>k_1$. 
Finally, since $(\|U_k\|_1)_k$ is a bounded sequence, 
$|\{|U_k|\ge R\}|\le\frac {\|U_k\|_1}R \le c\epsilon$.
From the inequality 
\[
|\sigma_k(x)-\sigma(x)|
\le
|F_k'(x,U_k)-F'(x,U_k)| + |F'(x,U_k)-F'(x,U)|
\]
we conclude when $k\ge \max\{k_0,k_1\}$ that 
\begin{align*}
\big\{ |\sigma_k-\sigma|>\epsilon\big\}
&\subset
\big\{ |F_k'(\cdot,U_k)-F'(\cdot,U_k)|>0\big\}  \cup 
\big\{|F'(\cdot,U_k)-F'(\cdot,U)|>\epsilon\big\}  \\
&\subset 
\big\{|U_k|\ge R\big\}  
\cup
\big\{ |U_k-U|\ge\delta\big\} \cup (\Omega\setminus K_\epsilon).
\end{align*}
Therefore $\big|\big\{ |\sigma_k-\sigma|>\epsilon\big\}\big|<(1+c)\epsilon$
when $k\ge \max\{k_0,k_1\}$ and so we have proved convergence in measure.

Earlier proofs for the convergence $\sigma_k\to\sigma$ 
used that $F_k\to F$ and $F_k'\to F'$ locally uniformly in $\xi$, 
but the proof offered for the latter in \cite[p.~1075]{CarKP14} was had a slight problem.

\begin{remark}\label{rem:correction}
One part of the four-stage approximation procedure of \cite{CarKP14, CarKP15, EleP22, KocK_pp, RicT24} 
is to show that locally uniform convergence of convex $C^1$-functions $G_k$ to $G$ 
implies the same for the derivatives. 
As a part of their proof that $G_k' \to G'$, 
Carozza, Kristensen and Passarelli di Napoli \cite[p.~1075]{CarKP14} state that 
\begin{align*}
\big|(G_k'(\xi_k)-G'(\xi)) \cdot \zeta\big|
&\le 
\bigg| \frac{G_k(\xi_k + t\zeta) - G_k(\xi_k)- t G'(\xi)\cdot \zeta}{t}\bigg| \\
&\le 
\big| G_k(\xi_k + \zeta) - G_k(\xi_k)-  G'(\xi)\cdot \zeta\big|
\end{align*}
for all $\zeta\in \Rn$ and $0<|t|\le 1$ ``because difference-quotients of convex functions are 
increasing in the increment''. However, this inequality does not hold. Let us 
take as a counter-example 
$G_k(\xi):=G(\xi):=\frac12|\xi|^2$. The claim with $\zeta:=\xi-\xi_k$ becomes 
\[
|\xi_k-\xi|^2 \le \big| \tfrac12\,|\xi|^2- \tfrac12\,|\xi_k|^2 - \xi\cdot (\xi-\xi_k)\big|
= \tfrac12 |\xi_k-\xi|^2,
\]
which is false when $\xi_k\ne\xi$. 
The same flawed argument it repeated in \cite{CarKP15, KocK_pp, RicT24}.
Nevertheless, this flaw is not fatal and can be rectified 
by not introducing absolute values prematurely. 
%
\end{remark}

\section{Proof of the main result}

In this section we prove Theorem~\ref{thm:main}. We first use $F_k$ from Section~\ref{sect:approximation} to construct auxiliary problems and show that their almost-minimizers converge. As a side-product this gives the existence of minimizers.
With appropriate limiting procedures following \cite{CarKP15}, 
this allows us to obtain the desired identities for the original minimizer. 
By construction, the functions $F_k$ have linear growth, so we cannot guarantee 
the existence of solutions to the approximating minimization problems
\[
\min_{w \, \in \, K} \int_\Omega F_k(x,\nabla w) \,dx.
\]
Instead, we use Ekeland variational principle, first proved in \cite{Eke74},  
to generate a sequence of functions $(u_k)_k$ that plays the same role. 

\begin{theorem}[Ekeland's variational principle, Theorem~1.4.1, page~29, \cite{Zal02}]
\label{Ekeland}
Let $(V, d)$ be a complete metric space and $\mathcal F : V \to \R$ be a lower semicontinuous function bounded from below. For $\epsilon > 0$ and $v \in V$ with 
$\mathcal F(v) \le \inf_V \mathcal F + \epsilon$,
there exists for every $\lambda > 0$ a unique minimizer $v_\lambda \in V$ 
of the functional 
\[
w \longmapsto \mathcal F(w) + \epsilon \lambda^{-1} d(w, v_\lambda)
\]
with $d(v,v_\lambda) \le \lambda$ and $\mathcal F(v_\lambda) \le \mathcal F(v)$.
\end{theorem}

\subsection*{Approximation of the minimizer}
We define
\[
I(w) := \int_\Omega F(x,\nabla w) \,dx 
\quad\text{and}\quad
I_k(w) := 
\int_\Omega F_k(x,\nabla w) \,dx.
\]
Fix a positive vanishing sequence $(\varepsilon_k)_k$. 
Let $(u_k)$ be the sequence of $\epsilon_k$-almost minimizers generated by Theorem~\ref{Ekeland} with
functional $I_k$, $\epsilon=\epsilon_k$, $\lambda=\sqrt{\epsilon_k}$, $V=(K, \|\cdot\|_{W^{1,1}})$ 
and arbitrary $v\in V$ with $\mathcal F(v) \le \inf_V \mathcal F + \epsilon$. 
By Lemma~\ref{lem:superlinear}, $M_{F^*}<\infty$. 
Using $F_k(x,\xi)\ge j\,|\xi| - M_{F^*}(j)$ with $j=1$, we get that $(\nabla u_k)_k$ is bounded in 
$L^1(\Omega;\Rn)$. Since $u_k\in K\subset W^{1,1}_{u_0}(\Omega)$, the 
Rellich--Kondrachov Theorem yields a (non-relabeled) subsequence $(u_k)_k$ converging to some $u$ in $L^1(\Omega)$.

Let $\epsilon>0$ and fix $j>\frac1\epsilon$ and $\lambda>2M_{F^*}(j)$. Then 
$t \le 2(t-M_{F^*}(j)) \le \frac2j F_k(x, t)$ when $t\ge \lambda$ by the earlier lower bound 
on $F_k$, so that 
\[
\int_{\{|\nabla u_k|>\lambda\}} |\nabla u_k|\, dx 
\le
\frac 2j\int_{\{|\nabla u_k|>\lambda\}} F_k(x, |\nabla u_k|)\, dx
< C\epsilon
\]
when $k\ge j$. For every $k<j$ we can choose $\lambda_k>0$ such that 
$\int_{\{|\nabla u_k|>\lambda_k\}} F_k(|\nabla u_k|)\, dx
< \epsilon$ by the the absolute continuity of the intergal. 
Then Theorem~1.3, (b)$\Rightarrow$(a), \cite[p.~239]{EkeT99} implies that 
$(\nabla v_k)_k$ has a subsequence which converges weakly in $L^1(\Omega, \R^n)$.
Since $K$ is convex and closed, it is weakly closed \cite[p.~6]{EkeT99}, 
and so $u\in K$. 

Since $F_k$ is convex and has linear growth in the second variable, 
the functional $I_k$ is weakly lower semicontinuous 
in $W^{1,1}_{u_0}(\Omega)$ \cite[Corollary~3.24, Remark~12.13]{Dac08}. 
Since $I_k$ is increasing in $k$, it follows from this that 
\[
I_k(u) 
\le 
\liminf_{j\to \infty} I_k(u_j)
\le 
\liminf_{j\to \infty} I_j(u_j)
\le
\liminf_{j\to \infty} \big(\inf_{w\in K} I_j(w) + \epsilon_j\big)
\le \inf_{w\in K} I(w). 
\]
Then it follows by monotone convergence that 
\[
I(u)
= 
\lim_{k\to \infty}
I_k(u) 
\le 
\inf_{w\in K} I(w) 
\]
and so $u$ is a minimizer of $I$. 
By the strict convexity, the minimizer of $I$ in $K$ is unique.

By the Ekeland variational principle construction, $u_k$ is the unique minimizer of
\[
w\longmapsto\int_\Omega  F_k(x,\nabla w)+\epsilon_k\,|\nabla w-\nabla u_k|  \,dx.
\]
For $w\in K$ and $0<h<1$, we define
\begin{align*}
w_k:=u_k+h(w-u_k).
\end{align*}
Since $w_k=(1-h)u_k+hw$, it belongs to $K$ by convexity. 
From the definition of $w_k$, we see that 
$\epsilon_k\,|\nabla w_k-\nabla u_k|= h\epsilon_k\,|\nabla w-\nabla u_k|$, 
so the minimizing property of $u_k$ implies that  
\begin{align*}
\int_\Omega \frac{F_k(x,\nabla u_k+h\nabla(w-u_k))-F_k(x,\nabla u_k)}h + \epsilon_k\, |\nabla w-\nabla u_k|\,dx 
\ge 0.
\end{align*}
Since $F_k$ is $k$-Lipschitz, $\frac{|F_k(x,\nabla w_k)-F_k(x,\nabla u_k)|}h \le \frac kh\,|\nabla(w_k-u_k)| 
= k\,|\nabla(w-u_k)|$, which is a suitable majorant. 
Since $F_k$ is differentiable, 
\[
\frac{F_k(x,\nabla u_k+h\nabla(w-u_k))-F_k(x,\nabla u_k)}h \to F_k'(x,\nabla u_k)\cdot \nabla(w-u_k)
\]
as $h\to 0^+$ for a.e.\ $x\in\Omega$. 
Thus we obtain by dominated convergence and the minimizing property that
\begin{align}
\label{dis-var}
\int_\Omega F_k'(x,\nabla u_k)\cdot \nabla(w-u_k) \,dx 
\ge -\epsilon_k\int_\Omega|\nabla w-\nabla u_k| \,dx
\ge -c\epsilon_k,
\end{align}
where $c:=\|\nabla w\|_1+\sup_k \|\nabla u_k\|_1 < \infty$.

\subsection*{The integrability claims}

Let $\sigma_k:=F_k'(x,\nabla u_k)$ and $\sigma:=F'(x,\nabla u)$, 
and recall that $F(\cdot, t\, \nabla w_0)\in L^1(\Omega)$. 
We choose $w := w_0$ in \eqref{dis-var} and obtain that
\begin{align*}
\int_\Omega \sigma_k\cdot \nabla w_0 \,dx + c\epsilon_k\ge\int_\Omega \sigma_k\cdot \nabla u_k \,dx.
\end{align*}
By 
Fenchel's identity, 
$F^*_k(x,\sigma_k) = \sigma_k\cdot \nabla u_k - F_k(x,\nabla u_k)
\le \sigma_k\cdot \nabla u_k$. 
Integrating over $\Omega$ and using the previous inequality, we have 
for $t > 1$ from the assumption of Theorem~\ref{thm:main} that
\begin{align*}
\int_\Omega F^*_k(x,\sigma_k) \,dx 
&\le \int_\Omega \sigma_k\cdot \nabla u_k \,dx 
\le 
\frac{1}{t} \int_\Omega \sigma_k\cdot t\nabla w_0 \,dx +c\epsilon_k  \\
&\le \frac{1}{t} \int_\Omega F^*_k(x,\sigma_k) \,dx + \, \frac{1}{t} \int_\Omega F_k(x, t \, \nabla w_0) \,dx +c\epsilon_k ,
\end{align*}
where we also used Fenchel's inequality. 
Reabsorbing the first term on the right-hand side into the left-hand side and using 
$F_k \le F$  and that $(\epsilon_k)_k$ is bounded, we obtain that
\begin{align*}
\int_\Omega F^*_k(x,\sigma_k) \,dx  
&\le \frac{1}{t-1} \int_\Omega F_k(x, t \, \nabla w_0) \,dx  +\frac{ct}{t-1} \epsilon_k  
\le C_1 \int_\Omega F(x,t \, \nabla w_0) \,dx +C_2.
\end{align*}
By $F_k\le F$ it follows from the definition that $F_k^* \ge F^*$. 
Since $F(\cdot, t \, \nabla w_0) \in L^1(\Omega)$ by assumption, we conclude that
\begin{equation}\label{non-fin}
\int_\Omega F^*(x,\sigma_k) \,dx 
\le 
C_1 \int_\Omega F(x,t \, \nabla w_0) \,dx + C_2 < \infty.
\end{equation}
Since $(\sigma_k)_k$ converges in measure to $\sigma$ (as proved at the end of 
Section~\ref{sect:approximation}), 
it converges almost everywhere (up to a non-relabeled subsequence).
Thus the previous estimate and Fatou's Lemma yield that
\[
\int_\Omega F^*(x,F'(x,\nabla u)) \,dx \le \liminf_{k \to \infty} \int_\Omega F^*(x, \sigma_k) \,dx < \infty.
\]
Hence $F^*(\cdot,F'(\cdot,\nabla u)) \in L^1(\Omega)$.
Since $F(\cdot, \nabla u) \in L^1(\Omega)$ by the definition of minimizers, Fenchel's 
identity gives $F'(\cdot,\nabla u)\cdot \nabla u \in L^1(\Omega)$.

\subsection*{The variational inequality}

Since $F$ is locally bounded, $M_F(t)$ is as well. 
By Lemma~\ref{lem:superlinear}, $M_F^*(t)$ is superlinear at infinity. 
Since $F(x,\xi)\le M_F(|\xi|)$, the opposite inequality holds 
for the conjugates. By \eqref{non-fin},
\[
\sup_{k \in \N} \int_\Omega M_F^*(|\sigma_k|) \,dx 
\le 
\sup_{k\in\N} \int_\Omega F^*(x,\sigma_k) \,dx 
< \infty,
\]
and so the de la Vall\`{e}e--Poussin Theorem \cite[p.~240]{EkeT99} implies the equi-integrability of 
$(\sigma_k)_k$. Since also $\sigma_k\to \sigma$ a.e., 
Vitali's Convergence Theorem \cite[p.~133]{Rud87}
yields $(\sigma_k)_k \to \sigma$ in $L^1(\Omega)$. 
With this we choose $w:= u_k + \eta$ for $\eta\in W^{1,\infty}(\Omega)$ in \eqref{dis-var} and obtain that
\begin{equation*}
\int_\Omega \sigma\cdot \nabla \eta \,dx
= 
\lim_{k\to\infty} \int_\Omega \sigma_k\cdot \nabla \eta \,dx
\ge 
\lim_{k\to\infty} c\epsilon_k
=
0 
\end{equation*}
provided that $w\in K$. This completes the proof of Theorem~\ref{thm:main}. 

\begin{remark}
Theorem \ref{thm:main} should hold also in the vectorial case, 
i.e.\ when $F:\Omega\times \R^{n\times d}\to \R$, with $d\in\N$, and $K\subset W^{1,1}(\Omega;\R^d)$. Indeed, in this case the necessary and sufficient condition for the weak lower semicontinuity of $I$ on $W^{1,1}(\Omega,\R^d)$ is that $F$ is quasiconvex (see \cite[Theorems~8.1 and 8.11]{Dac08}), which is implied by the strict convexity. On the other hand, in the vectorial case our result is not optimal since quasiconvexity is weaker than strict convexity. Moreover, our result is also not optimal in the vectorial case with $(p,q)$-growth or linear growth since in these cases quasincovexity is not a necessary and sufficient condition for the lower semicontinuity of $I$. See \cite{Anz85, BecSch15, BecBulG18, Bild03} for details.
\end{remark}

\begin{remark}
\label{equivalence}
Using the Fenchel identity \eqref{FenchelId} as in the proof, we see that 
the two integrability conditions in Theorem~\ref{thm:main} are equivalent, since $F(\cdot, \nabla u)\in L^1(\Omega)$. 
Furthermore, $F(\cdot, t\nabla v)\in L^1(\Omega)$ for some $t>1$ implies that 
$F'(\cdot,\nabla v)\cdot \nabla v \in L^1(\Omega)$ which in turn implies that 
$F(\cdot, \nabla v)\in L^1(\Omega)$; this follows from the point-wise inequality 
\[
F(\xi)\le \xi\cdot F'(\xi)
\le
\frac{F(t\xi) -F(\xi) }{t-1}
\le
\tfrac1{t-1} F(t\xi). 
\]
If $F$ is doubling, then these conditions are equivalent and the factor $t>1$ is not needed. 
\end{remark}

\appendix
\section{An alternative proof of differentiability}\label{app:approximation}

We provide an alternative proof to Theorem~\ref{thm:approximation} which 
is not based on $\inf$-convolution, which shows that $F_k\in C^1(\Rn)$. This proof is more geometric 
and direct, but not as elegant and also involves some extra assumptions. 

\begin{theorem}\label{thm:approximation2}
If $F\in C^1(\Rn)$ with $F(0)=0=F'(0)$ is strictly convex and superlinear at infinity, then 
$F_k\in C^1(\Rn)$.
\end{theorem}
\begin{proof}
Since $F$ is differentiable, strictly convex and superlinear at infinity, 
$F^*$ is strictly convex and $C^1$ \cite[Theorem~6.2.4]{NicP18}. 

The maximum of the function $z\mapsto \xi\cdot z - F^*(z)$ in the definition of $F_k$ is obtained at 
an interior point where its gradient vanishes or at a boundary point where its 
gradient is perpendicular to the boundary:
\[
\xi - (F^*)'(z)=0\quad\text{for }z\in B_k
\]
or
\[
\xi - (F^*)'(z)=s z\quad\text{for }z\in \partial B_k, s\ge 0.
\]
Here $s\ge 0$ because this gives the exterior normal and thus the maximum. 
Note that $(F^*)'(z) = (F')^{-1}(z)$ by \cite[p.~257]{NicP18}. For $z\in B_k$, the gradient 
vanishes when 
$z=((F^*)')^{-1}(\xi) = F'(\xi)$ so Fenchel's identity gives $F_k(\xi) = \xi\cdot F'(\xi) - F^*(F'(\xi))=F(\xi)$
and thus $F_k$ coincides with $F$ in $(F')^{-1}(B_k)$. 

Let us define $H:\Rn\to \overline{B_k}$ by 
$H(\xi) := \argmax_{z\in \overline{B_k}} (\xi\cdot z - F^*(z))$; this is well-defined 
since the function is strictly concave and $\overline{B_k}$ is convex. 
Then $F_k(\zeta) = \zeta\cdot H(\zeta)-F^*(H(\zeta))$ and  
$F_k(\xi) \ge \xi\cdot H(\zeta)-F^*(H(\zeta))$ and vica versa for $H(\xi)$ so that 
\[
(\xi-\zeta)\cdot H(\zeta)\le F_k(\xi)-F_k(\zeta) \le (\xi-\zeta)\cdot H(\xi).
\]
Since $|H|\le k$, this gives that $F_k$ is $k$-Lipschitz. Once we prove that  
$H$ is continuous, this inequality  implies that $F_k'(\xi)=H(\xi)$ and so $F_k\in C^1$. 

Define a continuous mapping $G:\partial B_k\times (0,\infty)\to\Rn\setminus \{0\}$ by
\[
G(z, r) := 
\begin{cases}
(F^*)'(r z)&\text{if } r\in (0,1), \\
(r-1)z+(F^*)'(z)&\text{if } r\in [1,\infty). \\
\end{cases}
\]
This amounts to solving the zeros of the derivative in $\xi$, so this is some kind 
of inverse to $H$, as we see later. 
Strict monotonicity of $(F^*)'$ implies that $G|_{\partial B_k\times (0,1]}$ is an injection.
For $\partial B_k\times [1,\infty)$, we consider the equation
\[
sz+(F^*)'(z) = tw+(F^*)'(w).
\]
Since $(F^*)'$ is monotone, we obtain that 
\[
0\ge 
\big((F^*)'(w)-(F^*)'(z)\big)\cdot(z-w)
=
(sz-tw)\cdot(z-w) = s|z|^2+t|w|^2-(s+t)z\cdot w. 
\]
If $z,w\in \partial B_k$, then this is only possible when $z=w$. From the original equation we then 
see that also $s=t$. Thus $G|_{\partial B_k\times [1,\infty)}$ is an injection. 
Suppose next that $z\in \partial B_k$, $w\in B_k$, $s>0$ and $t=0$. The inequality 
becomes $z\cdot w \ge |z|^2$, which is impossible. Having checked all combinations, 
we conclude that $G$ is an injection. 
It follows from Brouwer's theorem on the invariance of domain 
\cite[Theorem~VI~9]{HurW48} that $G^{-1}$ is continuous. 
Let $P:\partial B_k\times (0, \infty)\to \Rn$ be the projection 
(in scaled polar coordinates) onto the ball $\overline {B_k}$: 
\[
P(z,s)=
\begin{cases}
s z&\text{if } s\in (0,1), \\
z&\text{if } s\in [1,\infty). \\
\end{cases}
\]
Since $P$ and $G^{-1}$ are continuous, the same is true for $H=P\circ(G^{-1})$ in 
$\Rn\setminus \{0\}$. In the neighborhood $(F')^{-1}(B_k)$ of $0$ 
we showed that $H=F'$, so it is continuous. Thus $H$ is continuous in all of $\Rn$. 
\end{proof}


\section*{Acknowledgment}

We would like to thank the anonymous referees especially for some mathematical ideas 
that substantially improved and simplified the proofs.

The author A.T. is member of the Gruppo Nazionale per l’Analisi Matematica,
la Probabilità e le loro Applicazioni (GNAMPA) of the Istituto Nazionale di Alta Matematica (INdAM). The author has been partially supported through the INdAM-GNAMPA 2024 Project “Interazione ottimale tra la regolarità dei coefficienti e l’anisotropia del problema in funzionali integrali a crescite non standard” (CUP: E53C23001670001).



\begin{thebibliography}{99}

%

\bibitem{Anz85}
\textsc{G. Anzellotti:}
The Euler equation for functionals with linear growth,
\emph{Trans. Amer.\ Math.\ Soc.}~\textbf{290} (1985), no.~2, 483--501.


\bibitem{AzaM17}
\textsc{D. Azagra, C. Mudarra}:
Whitney extension theorems for convex functions of the classes $C^1$ and $C^{1,\omega}$, 
\emph{Proc. London Math. Soc.}~\textbf{114} (2017), no.~1, 133--158.

\bibitem{AzaM19}
\textsc{D. Azagra, C. Mudarra}:
Global geometry and $C^1$ convex extensions of 1-jets, 
\emph{Anal. PDE}~\textbf{12} (2019), no.~4, 1065--1099.
%

\bibitem{BecSch15}
\textsc{L. Beck, T. Schmidt:}
Convex duality and uniqueness for BV-minimizers. 
\emph{J.\ Funct.\ Anal.}~\textbf{268} (2015), no.~10, 3061--3107.

\bibitem{BecBulG18}
\textsc{L. Beck, M. Bulíček, F. Gmeineder:}
On a Neumann problem for variational functionals of linear growth,  
Preprint (2018). arXiv:1801.03014

\bibitem{Bild03}
\textsc{M. Bildhauer:}
\emph{Convex variational problems. Linear, nearly linear and
anisotropic growth conditions}, 
Lecture Notes in Mathematics, vol.~1818, Springer, Berlin, 2003.

\bibitem{CarKP14} 
\textsc{M. Carozza, J. Kristensen, A. Passarelli di Napoli}:
Regularity of minimizers of autonomous convex variationals intergrals, 
\emph{Ann. Sc. Norm. Super. Pisa Cl. Sci. (5)}~\textbf{13} (2014), no.~4, 1065--1089.

\bibitem{CarKP15} 
\textsc{M. Carozza, J. Kristensen, A. Passarelli di Napoli}:
On the validity of the Euler-Lagrange system, 
\emph{Comm. Pure Appl. Anal.}~\textbf{14} (2015), no.~1, 51--62.

\bibitem{ChlGSW21}
\textsc{I.\ Chlebicka, P.\ Gwiazda, A.\ \'{S}wierczewska-Gwiazda, A.\ Wr\'{o}blewska-Kami\'{n}ska}:
\emph{Partial Differential Equations in Anisotropic Musielak-Orlicz Spaces}, 
Springer, Cham, 2021.

\bibitem{ColM15} 
\textsc{M.\ Colombo, G.\ Mingione}: 
Regularity for double phase variational problems, 
\emph{Arch. Ration. Mech. Anal.}~\textbf{215} (2015), no. 2, 443--496.

\bibitem{Dac08}
\textsc{B.\ Dacorogna}: 
\emph{Direct Methods in Calculus of Variations}, 
Springer, New York, NY, 2008.


\bibitem{DefM23}
\textsc{C. De Filippis, G. Mingione}: 
Regularity for double phase problems at nearly linear growth, 
\emph{Arch. Ration. Mech. Anal.}~\textbf{247} (2023), article 85.


\bibitem{Eke74}
\textsc{I. Ekeland}:
On the Variational Principle,
\emph{J.\ Math. Anal. Appl.}~\textbf{47} (1974), 324--353.

\bibitem{EkeT99} 
\textsc{I. Ekeland, R. Temam}:
\emph{Convex Analysis and Variational Problems},
Classics in Applied Mathematics \textbf{28}, SIAM, Philadelphia, 1999.


\bibitem{EleP22}
\textsc{M. Eleuteri, A. Passarelli di Napoli:} 
On the validity of variational inequalities for obstacle problems with non-standard growth, 
\emph{Ann. Fenn. Math.}~\textbf{47} (2022), no.~1, 395--416.


\bibitem{GiaP13}
\textsc{F.\ Giannetti, A.\ Passarelli di Napoli}: 
Regularity results for a new class of functionals with non-standard growth conditions, 
\emph{J. Differential Equations}~\textbf{254} (2013), no.~3, 1280--1305.

\bibitem{HarH19}
\textsc{P. Harjulehto, P. H\"ast\"o}:
\emph{Orlicz Spaces and Generalized Orlicz Spaces}, 
Lecture Notes in Mathematics, vol. \textbf{2236}, Springer, Cham, 2019. 


\bibitem{Has23}
\textsc{P. H\"ast\"o}:
A fundamental condition for harmonic analysis in anisotropic generalized Orlicz spaces, 
\emph{J. Geom. Anal.}~\textbf{33} (2023), article 7.

\bibitem{HasO22}
\textsc{P. H\"ast\"o, J.\ Ok}: 
Regularity theory for non-autonomous partial differential equations without Uhlenbeck structure, 
\emph{Arch. Ration. Mech. Anal.}~\textbf{245} (2022), no.~3, 1401--1436.

\bibitem{HurW48}
\textsc{W.\ Hurewicz, H.\ Wallman}: 
\emph{Dimension Theory} (revised ed.), 
Princeton Mathematical Series, vol. \textbf{4}, Princeton University Press, Princeton, 1948.

\bibitem{KocK_pp}
\textsc{L. Koch, J. Kristensen}:
On the validity of the Euler-Lagrange system without growth assumptions, 
Preprint (2022). arXiv:2203.00333



\bibitem{NicP18}
\textsc{C.P. Niculescu, L.-E. Persson}:
\emph{Convex Functions and Their Applications: A Contemporary Approach} (2nd ed.), 
CMS Books in Mathematics, Springer, Cham, 2018.

\bibitem{RicT24}
\textsc{S.\ Riccò, A.\ Torricelli}:
A necessary condition for extremality of solutions to autonomous obstacle problems with general growth, 
\emph{Nonlinear Anal. Real World Appl.}~\textbf{76} (2024), article~104005.


\bibitem{Roc66}
\textsc{R.T. Rockafellar}: 
Extension of Fenchel's duality theorem, 
\emph{Duke Math. J.}~\textbf{33} (1966), 81--89.

\bibitem{Rud87} 
\textsc{W.\ Rudin}:
\emph{Real and Complex Analysis} (3rd ed.), 
McGraw--Hill, London, 1987.


\bibitem{Zal02} 
\textsc{C. Zalinescu:} 
\emph{Convex Analysis in General Vector Spaces}, 
World Scientific, River Edge, 2002.
\end{thebibliography}
\end{document}